\documentclass[a4paper,11pt]{amsart}
\usepackage{mathrsfs}
\usepackage{bbm}
\usepackage{latexsym, amssymb,amsmath,amsthm}
\usepackage[all, knot]{xy}
\xyoption{arc}

\def \To{\longrightarrow}

\def \gr{\operatorname{gr}}

\def \Corep{\operatorname{Corep}}
\def \unit{\epsilon}
\def \C{\mathcal{C}}
\def \D{\Delta}
\def \d{\delta}
\def \dl{\d_{_L}}
\def \dr{\d_{_R}}
\def \e{\varepsilon}

\def \N{\mathbb{N}}
\def \R{\mathcal{R}}
\def \S{\mathcal{S}}
\def \Z{\mathbb{Z}}

\numberwithin{equation}{section}

\newtheorem{theorem}{Theorem}[section]
\newtheorem{lemma}[theorem]{Lemma}
\newtheorem{proposition}[theorem]{Proposition}
\newtheorem{corollary}[theorem]{Corollary}

\newtheorem{remarks}[theorem]{Remarks}
\newtheorem{convention}[theorem]{Convention}

\begin{document}
\title[COQUASITRIANGULAR POINTED MAJID ALGEBRAS]
{ON COQUASITRIANGULAR POINTED \\ MAJID ALGEBRAS}
\author{Hua-Lin Huang}
\address{School of Mathematics, Shandong University, Jinan 250100, China} \email{hualin@sdu.edu.cn}
\author{Gongxiang Liu}
\address{Department of Mathematics, Nanjing University, Nanjing 210093, China} \email{gxliu@nju.edu.cn}
\date{}
\date{}
\maketitle

\begin{abstract}
We study coquasitriangular pointed Majid algebras via the quiver
approaches. The class of Hopf quivers whose path coalgebras admit
coquasitriangular Majid algebras is classified. The quiver setting
for general coquasitriangular pointed Majid algebras is also
provided. Through this, some examples and classification results are
obtained.
\end{abstract}

\section{Introduction}

Quasitriangular quasi-Hopf algebras were introduced and profoundly
studied by Drinfeld in a series of papers \cite{d2,d3,d4}. These are
a natural generalization of quasitriangular Hopf algebras which play
an essential role in his theory of quantum groups \cite{d1}. They
turn out to have deep connections with tensor categories, conformal
field theory, knot invariants, Grothendieck-Teichm\"{u}ller group,
multiple zeta value, and so on.

As far as we know, there are not many examples of quasitriangular
quasi-Hopf algebras in literature other than the quasi-triangular
quasi-Hopf QUE-algebras of Drinfeld \cite{d2} and the twisted
quantum double of finite groups of Dijgraaf-Pasquier-Roche
\cite{dpr} as well as their various generalizations. In particular,
the fundamental classification problem is still widely open. It is
our expectation that there will be a nice theory for the
classification problem of some interesting classes of
quasitriangular quasi-Hopf algebras and the associated braided
tensor categories, for instance an extension of the classification
theory of finite-dimensional triangular Hopf algebras over the field
of complex numbers due to Etingof and Gelaki (see \cite{eg} and
references therein) into the quasitriangular quasi-Hopf setting.

This paper is devoted to the study of quasitriangular quasi-Hopf
algebras via the quiver approaches initiated in \cite{qha1,qha2}. As
before we will work on a dual setting, namely the so-called
coquasitriangular Majid algebras, since this allows a wider scope
and the convenience of exposition. A recent work of the authors
\cite{hsaq5} shows that the coquasitriangularity of pointed Hopf
algebras can be described by combinatorial property of Hopf quivers.
Moreover, the quiver setting helps to give a complete classification
of finite-dimensional coquasitriangular pointed Hopf algebras over
an algebraically closed field of characteristic 0. The basic aim of
the present paper is to extend the study to the quasi situation.

We start by showing that the path coalgebra $kQ$ of a quiver $Q$
admits a coquasitriangular Majid algebra structure if and only if
$Q$ is a Hopf quiver of the form $Q(G,R)$ with $G$ abelian. Next we
give a classification of the set of graded coquasitriangular Majid
structures on a given connected Hopf quiver of this form. Then we
show that, for a general coquasitriangular pointed Majid algebra,
its graded version induced by coradical filtration can be viewed as
a large sub structure of a graded coquasitriangular Majid structure
on some unique Hopf quiver defined in the previous step. So far a
quiver setting for the class of coquasitriangular pointed Majid
algebras is built up. Finally we use the quiver setting to provide
some examples and classification results.

Throughout the paper, we work over a field $k.$ Vector spaces,
algebras, coalgebras, linear mappings, and unadorned $\otimes$ are
over $k.$ The readers are referred to \cite{d2,majid} for general
knowledge of quasi-Hopf and Majid algebras, and to \cite{ass} for
that of quivers and their applications to associative algebras and
representation theory. We turn to \cite{qha1,qha2} frequently for
definitions, notations and results of the quiver setting of Majid
algebras.

\section{Majid Algebras and Their Quiver Setting}

In this section we recall the definition of coquasitriangular Majid
algebras and the quiver framework of pointed Majid algebras for the
convenience of the readers.

\subsection{Coquasitriangular Majid Algebras}
A Majid algebra $H$ with associator $\Phi$ is said to be
coquasitriangular, if there is a convolution-invertible map $\R: H
\otimes H \To k$ such that
\begin{gather}
\R(ab , c)=\Phi(c_1 , b_1 , a_1)\R(a_2 , c_2)\Phi^{-1}(a_3 , c_3 ,
b_2) \\ \times \R(b_3 ,
c_4)\Phi(a_4 , b_4 , c_5), \nonumber \\
\R(a , bc)=\Phi^{-1}(b_1 , c_1 , a_1)\R(a_2 , c_2) \Phi(b_2 , a_3 ,
c_3) \\ \times \R(a_4
, b_3)\Phi^{-1}(a_5 , b_4 , c_4), \nonumber \\
b_1a_1\R(a_2 , b_2)=\R(a_1 , b_1)a_2b_2
\end{gather} for all $a,b,c \in H.$ Here and below we use the Sweedler sigma notation $\D(a)=a_1 \otimes
a_2$ for the coproduct and $a_1 \otimes a_2 \otimes \cdots \otimes
a_{n+1}$ for the result of the $n$-iterated application of $\D$ on
$a.$ The map $\R$ is called a coquasitriangular structure of $H.$ A
coquasitriangular Majid algebra $(H,\Phi,\R)$ is called cotriangular
if \begin{equation} \R(a , b)\R(b , a) = \e(a)\e(b)
\end{equation} for all $a,b \in H.$

\subsection{Hopf Quivers}
A quiver is a quadruple $Q=(Q_0,Q_1,s,t),$ where $Q_0$ is the set of
vertices, $Q_1$ is the set of arrows, and $s,t:\ Q_1 \longrightarrow
Q_0$ are two maps assigning respectively the source and the target
for each arrow. A path of length $l \ge 1$ in the quiver $Q$ is a
finitely ordered sequence of $l$ arrows $a_l \cdots a_1$ such that
$s(a_{i+1})=t(a_i)$ for $1 \le i \le l-1.$ By convention a vertex is
said to be a trivial path of length $0.$ Let $Q_n$ denote the set of
paths of length $n$ in $Q.$ There is a natural path coalgebra
structure on the path space $kQ$ with coproduct defined by splitting
of paths.

According to \cite{cr2}, a quiver $Q$ is said to be a Hopf quiver if
the corresponding path coalgebra $kQ$ admits a graded Hopf algebra
structure. Hopf quivers can be determined by ramification data of
groups. Let $G$ be a group and denote its set of conjugacy classes
by $\C.$ A ramification datum $R$ of the group $G$ is a formal sum
$\sum_{C \in \C}R_CC$ of conjugacy classes with coefficients in
$\mathbb{N}=\{0,1,2,\cdots\}.$ The corresponding Hopf quiver
$Q=Q(G,R)$ is defined as follows: the set of vertices $Q_0$ is $G,$
and for each $x \in G$ and $c \in C,$ there are $R_C$ arrows going
from $x$ to $cx.$

\subsection{Quiver Setting for Majid Algebras}

It is shown in \cite{qha1} that the path coalgebra $kQ$ admits a
graded Majid algebra structure if and only if the quiver $Q$ is a
Hopf quiver. Moreover, given a Hopf quiver $Q=Q(G,R),$ the set of
graded Majid algebra structures on $kQ$ with $kQ_0=(kG,\Phi)$ as
Majid algebras is in one-to-one correspondence with the set of
$(kG,\Phi)$-Majid bimodule structures on $kQ_1.$ \emph{In this paper
we always ignore the difference of various quasi-antipodes of a
Majid algebra, since they are essentially equivalent according to
Drinfeld \cite{d2} (Proposition 1.1).}

Recall that if $M$ is a $(kG, \Phi)$-Majid bimodule, then the
underlying bicomodule structure makes it a $G$-bigraded space
$M=\bigoplus_{g,h \in G} \ ^gM^h$ with $(g,h)$-isotypic component
$^gM^h=\{ m \in M | \d_{_L}(m)=g \otimes m, \ \d_{_R}(m)=m \otimes h
\} \ .$ While the quasi-bimodule structure maps satisfy the
following equalities:
\begin{gather}
e(fm)=\frac{\Phi(e,f,g)}{\Phi(e,f,h)}(ef)m,\\
(me)f=\frac{\Phi(h,e,f)}{\Phi(g,e,f)}m(ef),\\
(em)f=\frac{\Phi(e,h,f)}{\Phi(e,g,f)}e(mf),
\end{gather}
for all $e,f,g,h \in G$ and $m \in \ ^gM^h.$

A Majid algebra is said to be pointed, if its underlying coalgebra
is pointed. Given a pointed Majid algebra $(H, \Phi),$ let
$\{H_n\}_{n \ge 0}$ be its coradical filtration. Then the
corresponding coradically graded coalgebra $\gr(H) = H_0 \oplus
H_1/H_0 \oplus H_2/H_1 \oplus \cdots$ has an induced graded Majid
algebra structure with graded associator $\gr(\Phi)$ satisfing
$\gr(\Phi)(\bar{a},\bar{b},\bar{c})=0$ for all homogeneous
$\bar{a},\bar{b},\bar{c} \in \gr(H)$ unless they all lie in $H_0.$
In particular, $H_0$ is a sub Majid algebra and turns out to be the
group algebra $kG$ of the group $G=G(H),$ the set of group-like
elements of $H.$ In addition, the restriction of $\Phi$ to the
coradical of $H$ is a 3-cocycle on $G.$

For a coquasitriangular pointed Majid algebra $(H,\Phi,\R),$ let
$(\gr(H), \gr(\Phi))$ be as above. Define the function $\gr(\R):
\gr(H) \otimes \gr(H) \To k,$ for all homogeneous elements $g,h\in
\gr(H),$ by \[ \gr(\R)(g,h) = \left\{
                                        \begin{array}{ll}
                                          \R(g,h), & \hbox{if $g,h \in H_0$;} \\
                                          0, & \hbox{otherwise.}
                                        \end{array}
                                      \right. \]
Then we have the following easy fact which is useful later on.

\begin{lemma}
The coradically graded version $(\gr(H), \gr(\Phi), \gr(\R))$ is
still a coquasitriangular Majid algebra.
\end{lemma}

Thanks to the Gabriel type theorem in \cite{qha1} (Theorem 3.4), for
an arbitrary pointed Majid algebra $H,$ its graded version $\gr(H)$
can be realized as a large sub Majid algebra of some graded Majid
algebra structure on a unique Hopf quiver. By ``large" it is meant
the sub Majid algebra contains the set of vertices and arrows of the
Hopf quiver.

\subsection{Multiplication Formula for Quiver Majid Algebras}
It is shown in \cite{qha1} that the path multiplication formula of
graded Majid algebras on Hopf quivers can be given via quantum
shuffle product as in \cite{cr2}.

Suppose that $Q$ is a Hopf quiver with a necessary $kQ_0$-Majid
bimodule structure on $kQ_1.$ Let $p \in Q_l$ be a path. An $n$-thin
split of it is a sequence $(p_1, \ \cdots, \ p_n)$ of vertices and
arrows such that the concatenation $p_n \cdots p_1$ is exactly $p.$
These $n$-thin splits are in one-to-one correspondence with the
$n$-sequences of $(n-l)$ 0's and $l$ 1's. Denote the set of such
sequences by $D_l^n.$ Clearly $|D_l^n|={n \choose l}.$ For $d=(d_1,
\ \cdots, \ d_n) \in D_l^n,$ the corresponding $n$-thin split is
written as $dp=((dp)_1, \ \cdots, \ (dp)_n),$ in which $(dp)_i$ is a
vertex if $d_i=0$ and an arrow if $d_i=1.$ Let $\alpha=a_m \cdots
a_1$ and $\beta=b_n \cdots b_1$ be paths of length $m$ and $n$
respectively. Let $d \in D_m^{m+n}$ and $\bar{d} \in D_n^{m+n}$ the
complement sequence which is obtained from $d$ by replacing each 0
by 1 and each 1 by 0. Define an element
$$(\alpha  \beta)_d=[(d\alpha)_{m+n} (\bar{d}\beta)_{m+n}] \cdots
[(d\alpha)_1 (\bar{d}\beta)_1]$$ in $kQ_{m+n},$ where $[(d\alpha)_i
(\bar{d}\beta)_i]$ is understood as the action of $kQ_0$-Majid
bimodule on $kQ_1$ and these terms in different brackets are put
together by cotensor product, or equivalently concatenation. In
terms of these notations, the formula of the product of $\alpha$ and
$\beta$ is given as follows:
\begin{equation}
\alpha  \beta=\sum_{d \in D_m^{m+n}}(\alpha  \beta)_d \ .
\end{equation}

\section{Coquasitriangular Majid Algebras on Quivers}

In this section, we determine those quivers whose path coalgebras
admit coquasitriangular Majid algebra structures. A classification
of the set of graded coquasitriangular structures on such quivers is
also obtained.

\subsection{}
Our first step is to determine the condition on a quiver $Q$ such
that its path coalgebra $kQ$ admits a coquasitriangular Majid
algebra structure.

\begin{proposition}
Let $Q$ be a quiver. Then $kQ$ admits a coquasitriangular Majid
algebra structure if and only if $Q$ is a Hopf quiver of form
$Q(G,R)$ where $G$ is an abelian group and $R$ a ramification datum.
\end{proposition}

\begin{proof}
Assume that $Q$ is a quiver such that $kQ$ admits a
coquasitriangular Majid structure. By Lemma 2.1, we can assume that
the coquasitriangular Majid algebra is graded, namely both the
associator $\Phi$ and the coquasitriangular structure $\R$
concentrate at degree 0. Then by \cite{qha1} (Theorem 3.1), in the
first place $Q$ must be a Hopf quiver, say $Q(G,R).$ Note that
$kQ_0=kG$ is a group algebra and that $(kG,\Phi,\R)$ is a
coquasitriangular Majid algebra. Here $\Phi$ and $\R$ are understood
as their restriction to the degree 0 part. Now by (2.3) we have
\[ hg \R(g,h) = \R(g,h) gh \] for all $g,h \in G.$ Since $\R$ is convolution-invertible, one always has
$\R(g,h) \ne 0$ and then $gh=hg.$ This proves that $G$ is an abelian
group.

Conversely, assume that $Q$ is the Hopf quiver $Q(G,R)$ of some
abelian group $G$ with respect to a ramification datum $R.$ Then we
can take the trivial 3-cocycle $\Phi$ on $G,$ that is,
$\Phi(f,g,h)=1$ for all $f,g,h \in G,$ and then the
$(kG,\Phi)$-Majid bimodule structure on $kQ_1$ which corresponds to
the product of a set of trivial $kG$-modules. For more detail, see
\cite{qha2} (Theorem 3.3). That implies, for all $g \in G$ and
$\alpha \in Q_1,$ we have $g \alpha=\alpha g.$ By the product
formula given in Subsection 2.4, this gives rise to a commutative
graded Majid structure on $kQ.$ In fact this is even a commutative
Hopf algebra as the 3-cocycle $\Phi$ is trivial. Apparently $(kQ,
\Phi, \e \otimes \e)$ is a coquasitriangular Majid algebra.
\end{proof}

\subsection{}
Next we turn to classify the set of graded coquasitriangular Majid
algebra structures on a Hopf quiver of the form $Q(G,R)$ with $G$
abelian and $R=\sum_{g \in G} R_gg.$ By the Cartier-Gabriel
decomposition theorem for pointed Majid algebras \cite{qha1}
(Theorem 4.1), every graded Majid algebra on a general Hopf quiver
can be written as the crossed product of the sub structure on its
connected component containing the identity and a group algebra
possibly twisted by a 3-cocycle. Therefore in the following we can
assume without loss of generality that the quiver $Q(G,R)$ is
connected. By definition, it is clear that the Hopf quiver $Q(G,R)$
is connected if and only if the set $\{g \in G | R_g \ne 0\}$
generates the group $G.$

\begin{theorem}
Let $Q=Q(G,R)$ be a connected Hopf quiver with $G$ abelian and
$R=\sum_{g \in G} R_gg.$ Then the set of graded coquasitriangular
Majid algebra structures on $kQ$ with associator and
coquasitriangular structure concentrating at degree 0 and $Q_0 \cong
G$ as groups is in one-to-one correspondence with the set of pairs
$(\Phi,\R)$ in which $\Phi:G \times G \times G \To k$ is a 3-cocycle
such that
\begin{gather}
\frac{\Phi(eg,f,t)\Phi(g,e,t)\Phi(e,t,f)}{\Phi(eg,t,f)\Phi(g,t,e)\Phi(e,f,t)}
=\frac{\Phi(gt,e,f)\Phi(g,ef,t)\Phi(t,e,f)}{\Phi(g,e,f)\Phi(g,t,ef)},\\
\frac{\Phi(e,g,t)\Phi(eg,f,t)}{\Phi(f,g,t)\Phi(eg,t,f)}
=\frac{\Phi(e,gt,f)\Phi(e,fg,t)\Phi(g,f,t)}{\Phi(e,g,f)\Phi(g,t,f)\Phi(f,g,t)}
\end{gather}
for all $e,f,g \in G$ and $t \in G$ with $R_t \ne 0,$ and $\R: G
\times G \To k$ is a map such that
\begin{gather}
\R(f , gh)=\R(f , g)\R(f , h)\frac{\Phi(g,f,h)}{\Phi(g,h,f)\Phi(f,g,h)}, \\
\R(fg, h)=\R(f , h)\R(g ,
h)\frac{\Phi(h,f,g)\Phi(f,g,h)}{\Phi(f,h,g)}, \\
\R(g,h)\R(h,g)=1
\end{gather}
for all $f,g,h \in G.$
\end{theorem}

\begin{proof}
Assume that $(kQ, \Phi, \R)$ is a graded coquasitriangular Majid
algebra with $\Phi$ and $\R$ concentrating at degree 0 and $Q_0
\cong G.$ Then the restriction to degree 0 part, namely
$(kG,\Phi,\R),$ is again coquasitriangular. By definition, it is
clear that $\Phi$ is a 3-cocycle on $G$ and by (2.1)-(2.2) $\R$
satisfies
\begin{gather*}
\R(f , gh)=\R(f , g)\R(f , h)\frac{\Phi(g,f,h)}{\Phi(g,h,f)\Phi(f,g,h)}, \\
\R(fg, h)=\R(f , h)\R(g ,
h)\frac{\Phi(h,f,g)\Phi(f,g,h)}{\Phi(f,h,g)}
\end{gather*}
for all $f,g,h \in G.$ Next we verify (3.5). Choose any $g,h \in G$
with $R_gR_h \ne 0.$ Then in $Q$ there are arrows starting from the
unit $\unit$ of $G,$ say $\alpha: \unit \To g$ and $\beta: \unit \To
h.$ Then by (2.3) we have \[ \beta g \R(g,\unit) = \R(g,h) g \beta,
\quad \alpha h \R(h,\unit) = h \alpha \R(h,g). \] Here we have used
the fact that $\Phi$ and $\R$ concentrate at degree 0. By
(2.1)-(2.2) it is easy to deduce that $\R(g,\unit)=1=\R(\unit ,h)$
for any $g,h \in G.$ Hence we have
\[ \beta g = \R(g,h) g \beta, \quad \alpha h = h \alpha \R(h,g). \]
Now together with (2.3) and (2.8) we have
\begin{eqnarray*}
\beta \alpha &=& [\beta g][\alpha] + [h \alpha][\beta] = \R(g
, h)[g \beta][\alpha] + [h \alpha][\beta] \\
&=& \R(g , h)\alpha \beta = \R(g , h)[\alpha h][\beta] + \R(g , h)[g \beta][\alpha] \\
&=& \R(g , h)\R(h , g)[h \alpha][\beta] + \R(g , h)[g \beta][\alpha]. \\
\end{eqnarray*}
It follows that $\R(g , h)\R(h , g)=1.$ For any $f,g,h \in G$ with
$R_fR_gR_h \ne 0,$ we have
\begin{eqnarray*}
&&\R(f,gh)\R(gh,f) \\ &=&\R(f , g)\R(f ,
h)\frac{\Phi(g,f,h)}{\Phi(g,h,f)\Phi(f,g,h)}\R(g,f)\R(h,f)\frac{\Phi(f,g,h)\Phi(g,h,f)}{\Phi(g,f,h)}
\\ &=&1.
\end{eqnarray*}
As the Hopf quiver $Q$ is connected, all such $f,g,h$ run through a
generating set of $G,$ so (3.5) follows. Finally we prove
(3.1)-(3.2). If $R_t \ne 0,$ then in $Q$ there is an arrow $\alpha:
\unit \To t.$ For any $e,f,g \in G,$ by the definition of Majid
algebra (see e.g. \cite{qha1}) we have
\begin{gather*}
e(f(g\alpha))=\frac{\Phi(e,f,gt)}{\Phi(e,f,g)}(ef)(g\alpha),\\
((g\alpha)e)f=\frac{\Phi(g,e,f)}{\Phi(gt,e,f)}(g\alpha)(ef),\\
(e(g\alpha))f=\frac{\Phi(e,g,f)}{\Phi(e,gt,f)}e((g\alpha)f).
\end{gather*}
Since $\R$ is a coquasitriangular structure, by the first equation
and (2.3) all the terms of the last two equations can be written as
some scalars times $(efg)\alpha.$ By comparison of the scalars, one
has (3.1) and (3.2).

Conversely, we assume that $(\Phi,\R)$ is a pair satisfying
(3.1)-(3.5). Let $M$ be the $k$-space spanned by the set $\{ g\alpha
 | g \in G, \alpha \in Q_1 \ with \ s(\alpha)=\unit\}.$ Set $\dl
(g\alpha)=gt(\alpha) \otimes g\alpha $ and $\dr(g\alpha) =g\alpha
\otimes g.$ Then it is direct to verify that $(M,\dl,\dr)$ is a
$kG$-bicomodule and is isomorphic to $kQ_1.$ For each $f \in G,$
define \begin{equation} f(g\alpha) = \Phi(f,g,t(\alpha)) (fg)\alpha,
\, \,
(g\alpha)f=\frac{\R(f,gt(\alpha))}{\R(f,g)}\Phi(f,g,t(\alpha))(fg)\alpha.
\end{equation} We claim that this defines $(kG,\Phi)$-Majid bimodule
on $M,$ that is, (2.5)-(2.7) hold and the quasi-bimodule structure
is compatible with the bicomodule structure. By definition (3.6), we
have
\begin{gather*}
e(f(g\alpha))=\Phi(f,g,t(\alpha))e((fg)\alpha)=\Phi(f,g,t(\alpha))\Phi(e,fg,t(\alpha))(efg)\alpha,\\
\frac{\Phi(e,f,gt(\alpha))}{\Phi(e,f,g)}(ef)(g\alpha)=\frac{\Phi(e,f,gt(\alpha))}{\Phi(e,f,g)}
\Phi(ef,g,t(\alpha))(efg)\alpha.
\end{gather*}
Since $\Phi$ is a 3-cocycle, it follows that \[ e(f(g\alpha)) =
\frac{\Phi(e,f,gt(\alpha))}{\Phi(e,f,g)}(ef)(g\alpha). \] This is
(2.5). Similarly, by direct calculation one can show that (3.1) and
(3.2) imply respectively (2.6) and (2.7). It is clear that the
quasi-bimodule structure maps are bicomodule morphisms. Now by
\cite{qha1} (Proposition 3.3), the $(kG,\Phi)$-Majid bimodule
structure on $M$ can provide a graded Majid algebra structure on
$kQ$ where the associator is the trivial extension of $\Phi.$ That
is, set $\Phi(x,y,z)=0$ whenever one of $x,y,z$ lies out of $kQ_0.$
The map $\R$ is extended trivially in a similar manner. We claim
that $(kQ,\Phi,\R)$ is coquasitriangular. Since $\Phi$ and $\R$
concentrate at degree 0, the axioms (2.1)-(2.2) are direct
consequence of the conditions (3.3)-(3.4). It remains to verify
(2.3). We need to show that the following equation
\[ \beta \alpha  {\R}(s(a_1) , s(b_1))= \R(t(a_m) , t(b_n)) \alpha \beta \]
holds for all paths $\alpha=a_m \cdots a_1,\ \beta=b_n \cdots b_1.$
Here we use the convention: if $m=0,$ then $\alpha \in Q_0$ and
$t(\alpha)=\alpha=s(\alpha).$ When $l(\alpha) + l(\beta) \le 1,$ the
equation is obvious. Now let $\alpha=a_m \cdots a_1,\ \beta=b_n
\cdots b_1$ with $m+n>1.$ Then we have by the preceding cases and
the product formula (2.8) that
\begin{eqnarray*}
 \beta \alpha &=& \sum_{d \in D_n^{m+n}} [(d\beta)_{m+n} (\bar{d}\alpha)_{m+n}]
 \cdots [(d\beta)_1 (\bar{d}\alpha)_1] \\
 &=& \sum_{d \in D_n^{m+n}} [\frac{\R(t((\bar{d}\alpha)_{m+n}) ,
 t((d\beta)_{m+n}))}{\R(s((\bar{d}\alpha)_{m+n}) ,
 s((d\beta)_{m+n}))} (\bar{d}\alpha)_{m+n} (d\beta)_{m+n}] \cdots \\
 & & [\frac{\R(t((\bar{d}\alpha)_1) , t((d\beta)_1))}{R(s((\bar{d}\alpha)_1)
 , s((d\beta)_1))} (\bar{d}\alpha)_1 (d\beta)_1] \\
 &=& \frac{\R(t(a_m) , t(b_n))}{\R(s(a_1) , s(b_1))} \sum_{d \in D_n^{m+n}}
 [(\bar{d}\alpha)_{m+n} (d\beta)_{m+n}] \cdots [(\bar{d}\alpha)_1 (d\beta)_1] \\
 &=& \frac{\R(t(a_m) , t(b_n))}{\R(s(a_1) , s(b_1))} \sum_{d \in D_m^{m+n}}
 [(d\alpha)_{m+n} (\bar{d}\beta)_{m+n}] \cdots [(d\alpha)_1 (\bar{d}\beta)_1] \\
 &=& \frac{\R(t(a_m) , t(b_n))}{\R(s(a_1) , s(b_1))} \alpha \beta.
\end{eqnarray*} This is exactly the desired equation. Note that in the third equality we have used the fact
$t((d\beta)_i)=s((d\beta)_{i+1})$ for $i=1, \cdots, m+n-1.$ Now we
are done.
\end{proof}

\subsection{}
We conclude this section by some remarks.

\begin{remarks} Keep the assumptions and notations of Subsection
3.2.
\begin{enumerate}
  \item The coquasitriangular structure $\R$ is sort of a ``quasi" skew-symmetric bicharacter of the group $G.$
  Clearly, if $\Phi$ is trivial, then $\R$ is a usual skew-symmetric bicharacter. This is the usual Hopf case as
  given by Theorem 3.3 in \cite{hsaq5}. More generally, if $\Phi$ is a coboundary, then by a
suitable twisting, we can also go back to the Hopf case.
  \item The coquasitriangular structures constructed in the previous
theorem are actually cotriangular by (3.5). By Lemma 2.1, all
possible (not necessarily graded and concentrating at degree 0)
  coquasitriangular Majid algebra structures on Hopf quivers degenerate to
  cotriangular ones. This reduces the classification problem of
  general coquasitriangular Majid structures on Hopf quivers to a lifting procedure of the cotriangular ones.
\end{enumerate}
\end{remarks}

\section{Coquasitriangular Pointed Majid Algebras}

The aim of this section is to provide a quiver setting for general
coquasitriangular pointed Majid algebras. Some examples and
classification results are also provided via the quiver setting.

\subsection{}
The following is our main result which enables us to construct
coradically graded coquasitriangular pointed Majid algebras
exhaustively on Hopf quivers. This is a quasi analogue of Theorem
4.2 in \cite{hsaq5} and the proof is given by adjusting the argument
there into our situation.

\begin{theorem}
Let $(H,\Phi,\R)$ be a coquasitriangular pointed Majid algebra, and
as in Subsection 2.3 let $(\gr(H),\gr(\Phi),\gr(\R))$ denote its
graded version. Then there exist a unique Hopf quiver $Q=Q(G,R)$
with $G$ abelian and a graded coquasitriangular Majid algebra
structure $(kQ,\Psi,\mathfrak{R})$ with $\Psi$ and $\mathfrak{R}$
concentrating at degree 0 such that $(\gr(H),\gr(\Phi),\gr(\R))$ is
isomorphic to a large sub structure of $(kQ,\Psi,\mathfrak{R}).$
\end{theorem}

\begin{proof}
 Let $G$ denote the set of group-like elements of $H.$ Then the
coradical $H_0$ of $H$ is the group algebra $kG.$ By restricting the
associator $\Phi$ and the coquasitriangular structure $\R,$ one has
a sub coquasitriangular Majid algebra $(kG,\Phi,\R).$ As Proposition
3.1, we have immediately that $G$ is an abelian group and $\Phi$ is
a 3-cocycle on $G.$ By the Gabriel type theorem for pointed Majid
algebras \cite{qha1}, there exists a unique Hopf quiver $Q=Q(G,R)$
such that $(\gr(H),\gr(\Phi))$ can be viewed as a large sub Majid
algebra of the graded Majid structure $(kQ,\Psi)$ determined by the
$(kG,\Phi)$-Majid bimodule $H_1/H_0.$ Note that $\Psi$ is actually
the trivial extension of the 3-cocycle $\Phi$ on $G.$ Let
$\mathfrak{R}$ be the trivial extension of $\R:G \times G \To k.$ By
the same argument as in the proof of Theorem 3.2, one can show that
$(kQ,\Psi,\mathfrak{R})$ is a graded coquasitriangular Majid algebra
and the embedding $(\gr(H),\gr(\Phi)) \hookrightarrow (kQ,\Psi)$
respects the coquasitriangular structures. This completes the proof.
\end{proof}

\subsection{}
By the quasi analogue of the Cartier-Gabriel decomposition theorem
for pointed Majid algebras \cite{qha1}, we can focus on the
connected ones (that is, those Majid algebras whose quivers are
connected) without loss of generality. In that case, we can say more
about their graded version.

\begin{corollary}
Suppose that $(H,\Phi,\R)$ is a connected coquasitriangular pointed
Majid algebra. Then its graded version $(\gr(H),\gr(\Phi),\gr(\R))$
is cotriangular.
\end{corollary}

The proof is clear by Remarks 3.3 (2) and Theorem 4.1. More
generally, the graded version of a non-connected coquasitriangular
pointed Majid algebra can be written as the crossed product of a
cotriangular one (namely, its connected component containing the
identity) and a group algebra twisted by a 3-cocycle.

Recall that a tensor category is called pointed if its simple
objects are invertible. See \cite{eo} for more definitions and
results on finite tensor categories used below. We remark that the
preceding result also implies an interesting consequence for braided
pointed finite tensor categories with integral Frobenius-Perron
dimensions of objects. It is well-known that such tensor categories
indeed correspond to the corepresentation categories of
finite-dimensional coquasitriangular pointed Majid algebras. Thus
Corollary 4.2 implies for any braided pointed finite tensor category
$\mathcal{C}$ with integral Frobenius-Perron dimensions of objects,
its connected component ($\mathcal{C}$ is essentially governed by
its connected component, see \cite{qha3} for details) containing the
unit object is tensor equivalent to a deformation of a connected
\emph{symmetric} pointed finite tensor category.

\subsection{} For simplicity, \emph{we assume that the ground field $k$
is algebraically closed of characteristic 0 in the rest of the
paper.} As an example, let us consider the case of connected
coquasitriangular pointed Majid algebras over the cyclic group
$\Z_n=<g>$ of order $n>1.$ And we will see the condition
``coquasitriangular" is strong enough to make such pointed Majid
algebras to be twisting equivalent to Hopf algebras.

 First we recall a list
of 3-cocycles on $\Z_n$ as given in \cite{g}. Let $q$ be a primitive
root of unity of order $n.$ For any integer $i \in \N,$ we denote by
$i'$ the remainder of division of $i$ by $n.$ A list of 3-cocycles
on $\Z_n$ are
\begin{equation}
\Phi_s(g^{i},g^{j},g^{k})=q^{si(j+k-(j+k)')/n}
\end{equation}
for all $0 \le s \le n-1$ and $0 \le i,j,k \le n-1.$ Obviously,
$\Phi_s$ is trivial (i.e., cohomologous to a 3-coboundary) if and
only if $s=0.$

Let $R$ be a ramification datum of $\Z_n$ and let $Q$ denote the
associated Hopf quiver  $Q(\Z_n,R).$ Assume that $Q$ is connected.
By Theorem 3.2, the set of graded coquasitriangular Majid algebras
on $kQ$ with associator and coquasitriangular structure
concentrating at degree 0 is equivalent to the set of pairs
$(\Psi,\mathfrak{R})$ satisfying (3.1)-(3.5). Take $\Psi=\Phi_s$ for
some $0 \le s \le n-1$ as given in (4.1). By (3.5), we have
$\mathfrak{R}(g,g)^2=1.$ Using induction and (3.3), one has
\begin{equation}
1 = \mathfrak{R}(g,g^n) = \mathfrak{R}(g,g)^nq^{-s}. \end{equation}
We claim that this indeed implies that $s=0$. In fact, by
$\mathfrak{R}(g,g)^2=1$ we know that $\mathfrak{R}(g,g)^{n}=1$ or
$\mathfrak{R}(g,g)^{n}=-1$. By (4.2), the first case implies that
$q^{-s}=1$ and $s=0$. If $\mathfrak{R}(g,g)^{n}=-1$, then $n$ must
be odd. Also, (4.2) shows that $q^{-s}=-1$. Note that $q$ is an
$n$-th primitive root of unity and so $1=(q^{-s})^{n}=(-1)^{n}=-1$.
This is absurd. Thus we always have $\mathfrak{R}(g,g)^{n}=1$. This
claim means, $\Psi$ can be chosen only as a 3-coboundary and thus
such graded coquasitriangular Majid algebras must be twisting
equivalent to cotriangular Hopf algebras by Remarks 3.3.

Now, together with Corollary 4.2, the following assertion is clear.

\begin{proposition}
Assume that $(H,\Phi,\R)$ is a connected coquasitriangular pointed
Majid algebra with the set of group-likes equal to $\Z_n.$ Then its
graded version $(\gr(H),\gr(\Phi),\gr(\R))$ is twisting equivalent
to a cotriangular Hopf algebra.
\end{proposition}

As direct consequence, the  pointed Majid algebras $M_+(8), M_-(8)$
and $M(32)$ over $\Z_2$ given in \cite{qha2}, and $M(n,s,q)$ with $s
\ne 0$ over $\Z_n$ given in \cite{qha3} are not coquasitriangular
since they are nontrivial graded pointed Majid algebras, that is,
pointed Majid algebras which are not twisting equivalent to Hopf
algebras. Note that finite-dimensional connected graded cotriangular
pointed Hopf algebras over $\Z_n$ is completely classified by
Corollary 6.3 of \cite{hsaq5}. Therefore, finite-dimensional
connected graded cotriangular Majid algebras over $\Z_n$ are
essentially known by the previous proposition.

Of course, Proposition 4.3 also implies the corresponding
consequence on connected braided pointed finite tensor categories
whose invertible objects consisting of the cyclic group $\Z_n.$ In
particular, together with \cite{qha3,hsaq5} in a fairly
straightforward way, we get a classification result for braided
pointed tensor categories of finite type, i.e., in which there are
only finitely many indecomposable objects.

\begin{corollary}
Any connected braided pointed tensor category of finite type whose
simple objects all have Frobenius-Perron dimension 1 is tensor
equivalent to a deformation of $\Corep H$ where $H$ is a generalized
Taft algebra which can be presented by generators $g$ and $x$ with
relations \[ g^n=1, \quad x^2=0, \quad gx=-xg. \] Here $n$ is an
even integer and $\Corep H$ denotes the comodule category of $H.$
\end{corollary}
\begin{proof} Let $\mathcal{C}$ be a connected braided pointed
tensor category of finite type whose simple objects all have
Frobenius-Perron dimension 1. Thus we know that there is a connected
coquasitriangular pointed Majid algebra $H$ of finite
corepresentation type such that $\Corep H=\mathcal{C}$ (see, for
example, Subsection 4.2 in \cite{qha3}). All connected pointed Majid
algebras of finite corepresentation type have been classified in
\cite{qha3} and they are shown to be pointed Majid algebras over
$\mathbb{Z}_{n}$ for some $n\in \mathbb{N}$. Therefore, by
Proposition 4.3 one can assume that $\gr H$ is a connected
cotriangular pointed Hopf algebra of finite corepresentation type.
It is known that a connected graded pointed algebra of finite
corepresentation type is indeed a generalized Taft algebra (see
\cite{GLL}). Corollary 6.3 of \cite{hsaq5} shows that this algebra
must be of the form as given in this corollary.
\end{proof}

By quiver representation theory, such braided tensor categories are
well understood. In particular, their Auslander-Reiten quivers are
truncated tubes of height 2, see for instance \cite{ass}.

Finally, we remark that the knowledge of connected coquasitriangular
pointed Majid algebras over $\Z_n$ also sheds some light on the
general ones over finite abelian groups. It is clear that a general
Hopf quiver $Q(G,R)$ with $G$ abelian is consisting of various sub
quivers of form $Q(\Z_n, r).$ Therefore at least the local structure
of a general coquasitriangular Majid algebra is known. The remaining
task is the gluing of these local structures.

\section{Summary}

A quiver setting for coquasitriangular pointed Majid algebras is
built. It shows that the coquasitriangularity can be described by
some combinatorial properties of Hopf quivers. The quiver approaches
provide practical way to construct bundles of coquasitriangular
Majid algebras and braided tensor categories.

So far we have only dealt with the coradically graded case. In order
to extend our work to the non-graded situation, a proper deformation
theory of pointed Majid algebras is very much desirable. This task
seems more complicated than in the Hopf case, as the associator gets
involved.

\vskip 0.5cm

\noindent{\bf Acknowledgements:} The research was supported by the
NSFC grants (10601052, 10801069) and the SDNSF grants (YZ2008A05,
ZR2009AM012). The authors are grateful to the DAAD for financial
support which enabled them to visit the University of Cologne. They
would also like to thank their host Professor Steffen K\"{o}nig for
his kind hospitality. The second author is supported by the Japan
Society for the Promotion of Science under the item ``JSPS
Postdoctoral Fellowship for Foreign Researchers" and he thanks
Professor Akira Masuoka for stimulating discussions.


\begin{thebibliography}{99}

\bibitem{ass}
I. Assem, D. Simson, A. Skowronski, Elements of the Representation
Theory of Associative Algebras. Vol. 1. Techniques of Representation
Theory. London Mathematical Society Student Texts, 65. Cambridge
University Press, Cambridge, 2006.

\bibitem{cr2}
C. Cibils, M. Rosso, Hopf quivers, J. Algebra 254 (2002) 241-251.

\bibitem{dpr}
R. Dijkgraaf, V. Pasquier, P. Roche, Quasi Hopf algebras, group
cohomology and orbifold models, in: Recent Advances in Field Theory,
Annecy-le-Vieux, 1990, Nucl. Phys. B Proc. Suppl. 18B (1991) 60-72.

\bibitem{d1}
V.G. Drinfeld, Quantum groups, Proceedings of the International
Congress of Mathematicians, Vol. 1, 2 (Berkeley, Calif., 1986),
798-820, Amer. Math. Soc., Providence, RI, 1987.

\bibitem{d2}
V.G. Drinfeld, Quasi-Hopf algebras, Leningrad Math. J. 1 (1990)
1419-1457.

\bibitem{d3}
V.G. Drinfeld, On quasitriangular quasi-Hopf algebras and a group
closely connected with ${\rm Gal}(\overline Q /Q),$ Leningrad Math.
J. 2(4) (1991) 829-860.

\bibitem{d4}
V.G. Drinfeld, On the structure of quasitriangular quasi-Hopf
algebras, Funct. Anal. Appl. 26(1) (1992) 63-65.

\bibitem{eg}
P. Etingof, S. Gelaki, The classification of finite-dimensional
triangular Hopf algebras over an algebraically closed field of
characteristic 0, Mosc. Math. J. 3(1) (2003) 37-43.

\bibitem{eo}
P. Etingof, V. Ostrik, Finite tensor categories, Moscow Math. J. 4
(2004) 627-654.

\bibitem{g}
S. Gelaki, Basic quasi-Hopf algebras of dimension $n^3,$ J. Pure
Appl. Algebra 198(1-3) (2005) 165-174.

\bibitem{qha1}
H.-L. Huang, Quiver approaches to quasi-Hopf algebras, J. Math.
Phys. 50(4) (2009) 043501 9pp.

\bibitem{qha2}
H.-L. Huang, From projective representations to quasi-quantum
groups, submitted, arXiv:0903.1472.

\bibitem{hsaq5}
H.-L. Huang, G. Liu, On quiver-theoretic description for
quasitriangularity of Hopf algebras, submitted, arXiv:0910.2009.

\bibitem{qha3}
H.-L. Huang, G. Liu, Y. Ye, Quviers, quasi-quantum groups and finite
tensor categories, submitted, arXiv:0906.3415.

\bibitem{GLL} G. Liu, F. Li, Pointed Hopf algebras of finite corepresentation type and their
classification, Proc. Amer. Math. Soc.  135(3) (2007) 649-657.

\bibitem{majid}
S. Majid, Foundations of Quantum Group Theory, Cambridge University
Press, Cambridge, 1995.

\bibitem{mont1}
S. Montgomery, Hopf Algebras and Their Actions on Rings, CBMS
Regional Conf. Series in Math. 82, Amer. Math. Soc., Providence, RI,
1993.

\end{thebibliography}
\end{document}